\documentclass[11pt]{article}
\usepackage[utf8]{inputenc}
\usepackage[english]{babel}
\usepackage{mathtools}
\usepackage{bbold}
\usepackage{amsmath,amssymb,color,graphicx,bbm,amsthm,verbatim,soul,xcolor}
\usepackage{cancel}
\usepackage{tikz}
\usepackage{rotating}
\usetikzlibrary{decorations.pathreplacing}
\usepackage{graphicx,subfig}
\usepackage{float}

\usepackage{algpseudocode}
\usepackage{algorithm}
\usepackage[font=small,labelfont=bf]{caption}
\usepackage{caption}
\usepackage{url}
\newcommand {\R}{\mathbb{R}}

\newcommand{\sgn}{\mathrm{sgn}}
\newcommand {\grad}{\nabla}

\newcommand{\matlab}{{\sc Matlab}}

\newcommand{\beq}{\begin{equation}}
\newcommand{\eeq}{\end{equation}}

\newcommand{\LRA}{\Leftrightarrow }

\newtheorem{theorem}                {Theorem}
\newtheorem{corollary}  [theorem]   {Corollary}
\newtheorem{lemma}      [theorem]   {Lemma}

\newtheorem{assumption}         {Assumption}

\DeclarePairedDelimiter\ceil{\lceil}{\rceil}

\title{Analysis of the Gradient Method with an Armijo-Wolfe Line Search on a Class of \\Nonsmooth Convex Functions}
\author{Azam Asl\thanks{Courant Institute of Mathematical Sciences, New York University}
\and 
Michael L.~Overton\thanks{Courant Institute of Mathematical Sciences, New York University. Supported in
part by National Science Foundation Grant DMS-1620083.}
}

\date{\today}

\begin{document}
\maketitle

\begin{abstract}
It has long been known that the gradient (steepest descent) method may fail on nonsmooth problems, but the
examples that have appeared in the literature are either devised specifically to defeat a gradient or
subgradient method with an exact line search or are unstable with respect to perturbation of the initial point. 
We give an analysis of  the gradient method 
with steplengths satisfying the Armijo and Wolfe  
inexact line search conditions on the nonsmooth convex function $f(x) = a|x^{(1)}| + \sum_{i=2}^{n} x^{(i)}$.
We show that if $a$ is sufficiently large,
satisfying a condition that depends only on the Armijo parameter, then, 
when the method is initiated at any point $x_0 \in \R^n$ with $x^{(1)}_0\not = 0$, the iterates converge to a 
point $\bar x$ with $\bar x^{(1)}=0$, although $f$ is unbounded below. We also give conditions under which the iterates $f(x_k)\to-\infty$,
using a specific Armijo-Wolfe bracketing line search.
Our experimental results demonstrate that our analysis is
reasonably tight.

\end{abstract}



\section{Introduction}
\label{sec:intro}
First-order methods have experienced a widespread revival in recent years, as the number
of variables $n$ in many applied optimization problems has grown far too large to apply methods that require more
than $O(n)$ operations per iteration.  Yet many widely used methods, including limited-memory quasi-Newton 
and conjugate gradient methods, remain poorly understood on nonsmooth problems, and even the simplest
such method, the gradient method, is nontrivial to analyze in this setting.
Our interest is in methods with inexact line searches, since exact line searches are clearly out of the question when the 
number of variables is large, while methods with prescribed step sizes typically converge slowly, particularly
if not much is known in advance about the function to be minimized.

The gradient method dates back to Cauchy \cite{Cauchy}.
Armijo \cite{ARM66} was the first to establish convergence to stationary points of smooth functions 
 using an inexact line search with a simple ``sufficient decrease" condition. 
 Wolfe \cite{WOL69}, discussing line search methods for more general classes of methods,
 introduced a ``directional derivative increase" condition among several others.
The Armijo condition ensures that the line search step is not too large while the Wolfe 
condition ensures that it is not too small.  Powell \cite{POW76a} seems to have been the first to point out that
combining the two conditions leads to a convenient bracketing line search, noting also in another paper
\cite{POW76b} that use of the Wolfe condition ensures that, for quasi-Newton methods, 
the updated Hessian approximation is positive definite.  
Hiriart-Urruty and Lemarechal \cite[Vol~1, Ch.~11.3] {HULM} 
give an excellent discussion of all these issues, although they reference neither \cite{ARM66} nor 
\cite{POW76a} and \cite{POW76b}. They also comment (p.~402) on a surprising error in  \cite{Cauchy}.

Suppose that $f$, the function to be minimized, is a nonsmooth convex function.  
An example of \cite{WOL75} shows that the ordinary gradient method with an 
exact line search may converge to a non-optimal point, without encountering any points where 
$f$ is nonsmooth except in the limit. This example is
stable under perturbation of the starting point, but it does not apply when the line search is inexact.  
Another example given in \cite[vol.~1, p.~363]{HULM} 
applies to a subgradient 
method in which the search
direction is defined by the steepest descent direction, i.e., the negative of 
the element of the subdifferential with smallest norm, again showing that use of an exact line search
results in convergence to a non-optimal point.
This example is also stable under perturbation of the initial point, and, unlike Wolfe's example, 
it also applies when an inexact line search is used, but it is more complicated than is needed for the 
results we give below because it was specifically designed to defeat the steepest-descent subgradient 
method with an exact line search.  Another example of convergence to a non-optimal point
of a convex max function using a specific subgradient method with an exact line search
goes back to \cite{DM71}; see \cite[p.~385]{Flet87}. 
More generally, in a ``black-box" subgradient method, 
the search direction is the negative of any subgradient returned by an ``oracle", which may 
not be a descent direction if the function is not differentiable at the point, 
although this is unlikely if the current point was not generated by an exact line search
since convex functions are differentiable almost everywhere.
The key advantage of the subgradient method is that, as long as $f$ is convex and bounded below, 
convergence to its minimal value can be guaranteed even if $f$ is nonsmooth by predefining
a sequence of steplengths to be used, but the disadvantage is that convergence is usually slow.
Nesterov \cite{Nes05} improved the complexity of such methods using a smoothing technique, 
but to apply this one needs some knowledge of the structure of the objective function.  

The counterexamples mentioned above
motivated the introduction of bundle methods by \cite{LEM75} and \cite{WOL75} for nonsmooth
convex functions and, for nonsmooth, nonconvex problems, the bundle methods of \cite{KIW85} and  
the gradient sampling algorithms of \cite{BLO} and \cite{KIW07}.
These algorithms all have fairly strong convergence properties, to a nonsmooth (Clarke)
stationary value when these exist in the nonconvex case (for gradient sampling, with probability one), 
but when the number of variables is large the cost per iteration is much higher than the cost of a gradient step.
See the recent survey paper \cite{BCLOS} for more details. The ``full" BFGS method is a very effective alternative
choice for nonsmooth optimization \cite{LO13},
and its $O(n^2)$ cost per iteration (for the matrix-vector products that it requires)
is generally much less than the cost of the bundle or gradient sampling methods, 
but its convergence results for nonsmooth functions are limited to very special cases. The limited memory
variant of BFGS \cite{LN89}
costs only $O(n)$ operations per iteration, like the gradient method, but its behavior on
nonsmooth problems is less predictable.

In this paper we analyze the ordinary gradient method 
with an inexact line search applied to a simple nonsmooth convex function.
We require points accepted by the line search to satisfy both Armijo and Wolfe conditions
for two reasons. The first is that our longer-term goal is to carry out a related analysis for the  
limited memory BFGS method for which the Wolfe condition is essential.
The second is that although the Armijo condition is enough to prove convergence of the gradient method
on smooth functions, the inclusion of the Wolfe condition is potentially useful in the nonsmooth case, where the
norm of the gradient gives no useful information such as an estimate of the distance to a minimizer.
For example, consider the absolute value function in one variable initialized at $x_0$ with $x_0$ large.
A unit step gradient method with only an Armijo condition will require $O(x_0)$ iterations just to change the
sign of $x$, while an Armijo-Wolfe line search with extrapolation defined by doubling requires only one
line search with $O(\log_2(x_0))$ extrapolations to change the sign of $x$. Obviously, the so-called strong
Wolfe condition recommended in many books for smooth optimization, which requires a reduction in the 
absolute value of the directional derivative, is a disastrous choice when $f$ is nonsmooth. We mention here
that in a recent paper on the analysis of the gradient method with fixed step sizes \cite{TayEA}, Taylor et al.\ remark that 
``we believe it would be interesting to analyze [gradient] algorithms involving line-search, such as backtracking or Armijo-Wolfe 
procedures."

We focus on the nonsmooth convex function mapping $\R^n$ to $\R$ defined by

\beq \label{fdef}
f(x) = a|x^{(1)}| + \sum_{i=2}^{n} x^{(i)} .
\eeq

We show that if $a$ satisfies a lower bound that depends only on the Armijo parameter, then the 
iterates generated by the gradient method with steps satisfying Armijo and Wolfe conditions converge to a point $\bar x$ with $\bar x^{(1)}=0$, regardless of the starting point, although $f$ is unbounded below.   
The function $f$ defined in \eqref{fdef} was also used by \cite[p.~136]{LO13} with
$n=2$ and $a=2$ to illustrate 
failure of the gradient method with a specific line search, but the observations made there are not stable with respect to
small changes in the initial point.

The paper is organized as follows. In Section~\ref{sec:no_specific_linesch} we establish the main theoretical results, without assuming the use of
any specific line search beyond satisfaction of the Armijo and Wolfe conditions. In Section~\ref{sec:specific_linesch}, we extend these results
assuming the use of a bracketing line search that is a specific instance of the ones
outlined by \cite{POW76a} and \cite{HULM}. 
In Section~\ref{sec:expt_results}, we give experimental results, showing that our theoretical results are
reasonably tight.
We discuss connections with the convergence theory for subgradient methods in Section~\ref{sec:subgradient}.
We make some concluding remarks in Section~\ref{sec:conclusion}.


\section{Convergence Results Independent of a Specific Line Search}\label{sec:no_specific_linesch}

First let $f$ denote any locally Lipschitz function mapping $\R^n$ to $\R$, and let
$x_k \in \R^n$, $k=0,1,\ldots,$ denote the $k$th iterate of an optimization algorithm
where $f$ is differentiable at $x_k$ with gradient $\grad f(x_k)$.  Let $d_k\in\R^n$
denote a descent direction at the $k$th iteration, i.e., satisfying $\grad f(x_k)^T d_k < 0$,
and assume that $f$ is bounded below on the line $\{x_k+t d_k: t \geq 0\}$.
Let $c_1$ and $c_2$, respectively the Armijo and Wolfe parameters,
satisfy $0<c_1<c_2< 1.$
We say that the step $t$ satisfies the Armijo condition at iteration $k$ if 
\beq
       A(t):\quad           f(x_k + t d_k) \leq f(x_k) + c_1 t \grad f(x_k)^T d_k \label{armijo_cond}
\eeq
and that it satisfies the Wolfe condition if \footnote{There is a subtle distinction between the Wolfe condition given here
and that given in \cite{LO13},
since here the Wolfe condition is understood to fail if the gradient of $f$ does not exist at $x_k+td_k$, while in \cite{LO13} it is
understood to fail if the function of one variable $s\mapsto f(x_k + s d_k)$ is not differentiable at $s=t$.
For the example analyzed here, these conditions are equivalent.} 
\beq
       W(t):\quad  f \mathrm{~is~differentiable~at~}x_k + t d_k \mathrm{~with~}  \grad f(x_k + t d_k)^T d_k \geq c_2 \grad f(x_k)^T d_k. \label{wolfe_cond}
\eeq
The condition $0<c_1<c_2< 1$ ensures that points $t$ satisfying $A(t)$ and $W(t)$ exist, as is well known in the convex case
and the smooth case; for more general $f$, see \cite{LO13}.
The results of this section are independent of any choice of line search to generate such points.
Note that as long as $f$ is differentiable at the initial iterate, defining subsequent
iterates by $x_{k+1}=x_k+t_k d_k$, where $W(t)$ holds for $t=t_k$, ensures that $f$ is differentiable at all $x_k$.

We now restrict our attention to $f$ defined by \eqref{fdef}, with
\beq\label{ddef}
 d_k = -\grad f(x_{k}) =  -\left[\begin{array}{c} \sgn(x_k^{(1)}) a  \\ \mathbb{1}\end{array}\right],
\eeq
where $\mathbb{1}\in\R^{n-1}$ denotes the vector of all ones.
We have
\[
      f(x_k + t d_k) = a\left|x_k^{(1)} - \sgn(x_k^{(1)})a t\,\right| + \sum_{i=2}^{n}x^{(i)}_k - (n-1)t.
\]
We assume that $a \geq \sqrt{n-1}$, so that $f$ is bounded below along the negative gradient direction as $t\to\infty$. Hence, $x_{k+1}= x_k + t_k d_k$ satisfies
\beq \label{osc}
         x_{k+1}^{(1)} = x_k^{(1)} - \,\sgn(x_k^{(1)})a t_k  \quad \mathrm{and} \quad x^{(i)}_{k+1} = x^{(i)}_k -t_k \mathrm{~~for~~} i=2,\hdots,n.
\eeq
We have 
\beq \label{gradf_d}
          \grad f(x_k)^T d_k = -(a^2 + n-1)
\eeq
and 
\beq \label{sgnn}
          \grad f(x_k + t_k d_k)^T d_k = -(a^2\,\sgn(x_{k+1}^{(1)})\sgn(x_k^{(1)}) + n-1).
\eeq

For clarity we summarize the underlying assumptions that apply to all the results in this section.
\begin{assumption}
Let $f$ be defined by \eqref{fdef} with $a \geq \sqrt{n-1}$ and define $x_{k+1}=x_k+t_k d_k$, with $d_k=-\grad f(x_k)$,
for some step $t_k$, $k=1,2,3,\ldots$, where $x_0$ is arbitrary provided that $x^{(1)}_0\not=0$.
\end{assumption}

\begin{lemma} 
\label{armijo_wolfe_equiv}
The Armijo condition $A(t_k)$ (i.e., \eqref{armijo_cond} with $t=t_k$), is equivalent to
\beq\label{arm_simp}
c_1t_k (a^2+n-1)  \le f(x_k) - f(x_{k+1})
\eeq 
and the Wolfe condition $W(t_k)$ (i.e., \eqref{wolfe_cond} with $t=t_k$) is equivalent to each of the following three conditions:
\beq
    \sgn(x_{k+1}^{(1)}) = -\sgn(x_k^{(1)})  \label{wolfe1},
\eeq
\beq
     t_k > \dfrac{|x_k^{(1)}|}{a} \label{wolfe2}
\eeq
and
\beq
         at_k=|x_{k+1}^{(1)} - x_k^{(1)}|= |x_k^{(1)}|+|x_{k+1}^{(1)}|. \label{wolfe3}
\eeq

\begin{proof} 
These all follow easily from \eqref{osc}, \eqref{gradf_d} and \eqref{sgnn}, using $c_2 < 1$ and  $a \geq \sqrt{n-1}$.
\end{proof}
\end{lemma}
Thus, $t_k$ satisfies the Wolfe condition if and only if the iterates $x_k$ oscillate back and forth
across the $x^{(1)}=0$ axis.\footnote{
The same oscillatory behavior occurs if we replace the Wolfe condition by the Goldstein condition
$f(x_k + t d_k) \ge f(x_k) + c_2 t \grad f(x_k)^T d_k$. }

\begin{theorem}
\label{key_ineq} 
Suppose $t_k$ satisfies $A(t_k)$ and $W(t_k)$ for $k=1,2,3,\ldots,N$ and define $S_N=\sum_{k=0}^{N-1} t_k$. 
Then
\beq \label{f_lb_ub}
     c_1 (a^2+n-1)S_N \le f(x_0)-f(x_N) \le (n-1)S_N + a|x^{(1)}_0|,
\eeq
so that $S_N$ is bounded above as $N\to\infty$ if and only if $f(x_N)$ is bounded below.  
Furthermore, $f(x_N)$ is bounded below if and only if $x_N$ converges to a point $\bar x$ with $\bar x^{(1)}=0$.

\end{theorem}
\begin{proof} 
Summing up \eqref{arm_simp} from $k=0$ to $k= N-1$ we have
 \beq\label{sumt}
c_1 (a^2+n-1)  S_N   \le f(x_0) - f(x_N).
\eeq 
Using \eqref{osc}  we have 
\[   
x^{(i)}_{0} - x^{(i)}_{N} = \sum^{N-1}_{k=0} (x^{(i)}_{k} - x^{(i)}_{k+1}) = S_N \mathrm{~~for~~} i=2,\hdots,n,
\]   
so
\[ f(x_0) -f(x_N)  =  a|x^{(1)}_0|-a|x^{(1)}_N| + (n-1)S_N .\]
Combining this with \eqref{sumt} and dropping the term $a|x^{(1)}_N|$ we obtain \eqref{f_lb_ub},
so $S_N$ is bounded above if and only if $f(x_N)$ is bounded below.  Now suppose that $f(x_N)$ is bounded below and hence $S_N$ is bounded
above, implying that $t_N\to 0$, and therefore, from \eqref{wolfe3}, that $x^{(1)}_N\to 0$.
Since $f(x_N)=a|x^{(1)}_N| + \sum^{n-1}_{i=2} x^{(i)}_{N}$ 
is bounded below as $N\to\infty$, and since, from \eqref{osc}, for $i=2,\ldots,n$, each $x^{(i)}_N$ is decreasing as $N$ increases,
we must have that each $x^{(i)}_N$ converges to a limit $\bar x^{(i)}$.
On the other hand, if $x_N$ converges to a point $(0,\bar x^{(2)}, \hdots,\bar x^{(n)})$ then $f(x_N)$ is bounded below by  $ \sum^{n-1}_{i=2}  \bar x^{(i)} $.
\end{proof}
Note that, as $f$ is unbounded below, convergence of $x_N$ to a point $(0,\bar x^{(2)}, \hdots,\bar x^{(n)})$ should be interpreted as failure of the method.

We next observe that, because of the bounds \eqref{f_lb_ub}, it is not possible that $S_N\to\infty$ if
\[
   a > \sqrt{(n-1)(1/c_1 -1)}  
\]
(in addition to $a\geq \sqrt{n-1}$ as required by Assumption 1).

It will be convenient to define 
\beq \label{taudef}
\tau =c_1+\dfrac{(n-1)(c_1-1)}{a^2} .
\eeq
Since $c_1\in(0,1)$ and $a \geq \sqrt{n-1}$,
 we have $ -1 < -1+2c_1 < \tau < c_1 < 1$, with $\tau > 0$ equivalent to  $c_1(a^2+n-1) > n-1$.

\begin{corollary}\label{tau_implications}
Suppose $A(t_k)$ and $W(t_k)$ hold for all $k$. If $\tau > 0$ then $f(x_k)$ is bounded below as $k\to\infty$.
\end{corollary}
\begin{proof}
This is now immediate from \eqref{f_lb_ub} and the definition of $\tau$. 
\end{proof}
So, the larger $a$ is, the smaller the Armijo parameter $c_1$ must be in order to have $\tau\leq 0$ and therefore the possibility that $f(x_k)\to-\infty$.

At this point it is natural to ask whether  $\tau \le 0$ implies that $f(x_k)\to -\infty$. We will see in the next section (in Corollary \ref{alg1_tau0}, for $\tau=0$)
that the answer is no. However, we can show that there is a specific choice of $t_k$ 
satisfying $A(t_k)$ and $W(t_k)$ for which $\tau\le 0$ implies $f(x_k)\to-\infty$. We start with a lemma.

\begin{lemma} 
\label{armijo_equiv}
 Suppose  $W(t_k)$ holds. Then $A(t_k)$ holds if and only if
 
\beq 
\label{armu}
(1+\tau)\dfrac{at_k}{2} \le |x_k^{(1)}| .
\eeq 
\end{lemma}
\begin{proof} 

Suppose $x^{(1)}_{k}  > 0$. Since $W(t_k)$ holds, we can rewrite the Armijo condition \eqref{arm_simp} as
\begin{align*}
c_1t_k(a^2+n-1) & \le f(x_k) - f(x_{k+1})\\
& =   \left(a x_k^{(1)} + \sum_{i=2}^{n}x^{(i)}_k\right) -\left(-a(x_k^{(1)}-at_k)+ \sum_{i=2}^{n}x^{(i)}_k - (n-1)t_k\right)\\
& \LRA  t_k\Big(c_1(a^2+n-1)+a^2 -(n-1)\Big)   \le   2a x_k^{(1)}  \nonumber \\
& \LRA  t_ka^2(\tau+1)   \le   2a x_k^{(1)}  \nonumber,
\end{align*}
giving \eqref{armu}.
A similar argument applies when $x_k^{(1)} < 0$.
\end{proof}

\begin{theorem} \label{tauneg_fdiverges}
Let
\beq   
      t_k = \dfrac{2|x_k^{(1)}|}{(\tau+1)a}.  \label{tkdef}
\eeq
Then  

(1) $A(t_k)$ and $W(t_k)$ both hold. 

(2) if $\tau \le 0$, then $f(x_k)$ is unbounded below as $k\to\infty$.
\end{theorem}
\begin{proof}
The first statement follows immediately from \eqref{wolfe2} (since $|\tau|<1$) and Lemma \ref{armijo_equiv}. Furthermore, \eqref{wolfe3} allows us to write \eqref{armu} equivalently as
 \beq \label{eq:taup}
( 1+\tau) |x_{k+1}^{(1)}| \le  (1-\tau)|x_k^{(1)}|.
\eeq
 Since $t_k$ is the maximum steplength satisfying \eqref{armu}, it follows that \eqref{eq:taup} holds with equality, so 
 $|x_{k+1}^{(1)}| = C|x_k^{(1)}|$, where 
\[
       C=\dfrac{1-\tau}{1+\tau},
\]
and hence 
$$|x_{k+1}^{(1)}| = C^{k+1}|x^{(1)}_0|.$$
Then, we can rewrite \eqref{tkdef} as 
\[
      t_k =\dfrac{2C^k|x^{(1)}_0|}{a(\tau+1)}.
\]  
When $-1< \tau\leq 0$, we have $C\geq 1$,  
so $S_N =\sum_{k=0}^{N-1} t_k \to\infty$ as $N\to\infty$ and hence, by Theorem \ref{key_ineq}, $f(x_N)\to-\infty$.
\end{proof}

\color{black}
\section{Additional Results Depending on a Specific Choice of Armijo-Wolfe Line Search}\label{sec:specific_linesch}

In this section we continue to assume that $f$ and $d_k$ are defined by \eqref{fdef} and \eqref{ddef}
respectively, with $a \geq \sqrt{n-1}$,
and that $A(t)$ and $W(t)$ are defined as earlier. However, unlike in the previous section, we now
assume that $t_k$ is generated by a specific line search, namely the one given in Algorithm~1,
which is taken from \cite[p.~147]{LO13} and is a specific realization of the line searches described implicitly in \cite{POW76a} and explicitly in \cite{HULM}.
Since the line search function $s\mapsto f(x_k+s d_k)$ 
is locally Lipschitz and bounded below, it follows, as shown in \cite{LO13},
that at any stage during the execution of Algorithm~1, the
interval $[\alpha,\beta]$ must always contain a set of points $t$ with nonzero measure satisfying $A(t)$ and $W(t)$,
and furthermore, the line search must terminate at such a point. This defines the point $t_k$.
A crucial aspect of Algorithm~1 is that, in the ``while" loop, the Armijo condition is tested first and the Wolfe
condition is then tested only if the Armijo condition holds.

 \begin{algorithm}
 \caption{(Armijo-Wolfe Bracketing Line Search)}
 \label{alg1}
\begin{algorithmic}
\State $\alpha \leftarrow 0$
\State $\beta \leftarrow +\infty $
\State $t\leftarrow 1$
\While  {true} 
      \If {$A(t) $ fails (see \eqref{armijo_cond})}
            \State $\beta \leftarrow t $
      \ElsIf  {$ W(t) $ fails (see \eqref{wolfe_cond})}
           \State $\alpha \leftarrow t $
      \Else
           \State  stop and return $t$
     \EndIf
     
     \If {$ \beta < +\infty $}
            \State $t \leftarrow (\alpha+\beta)/2$
      \Else
           \State $t \leftarrow 2\alpha$
     \EndIf
\EndWhile
\end{algorithmic}
\end{algorithm}

We already know from Theorem~\ref{key_ineq} and Corollary~\ref{tau_implications} that, for \emph{any} set of Armijo-Wolfe points, if $\tau > 0$, then $f(x_N)$ is bounded below.
In this section we analyze the case $\tau\leq 0$, assuming that the steps $t_k$ are generated by the Armijo-Wolfe bracketing line  search.
It simplifies the discussion to make a probabilistic analysis, assuming
that $x_0=(x^{(1)}_0,x^{(2)}_0, \hdots, x^{(n)}_0)$ is generated randomly, say from the normal distribution.   
Clearly, all intermediate values $t$ generated by Algorithm~1 are rational, and with probability one all corresponding points $x=(x^{(1)}_0 - \sgn(x^{(1)}_0)at, x^{(2)}_0 - t,\hdots,x^{(n)}_0 - t)$ where the Armijo and Wolfe
conditions are tested during the first line search are irrational (this is obvious if $a$ is rational but it also holds if $a$ is irrational assuming that
$x_0$ is generated independently of $a$). It follows that, with probability one, $f$ is differentiable at these points, which include the next iterate $x_1=(x^{(1)}_1,x^{(2)}_1, \hdots, x^{(n)}_1)$.
It is clear that, by induction, the points $x_k=(x_k^{(1)},x^{(2)}_k, \hdots, x^{(n)}_k)$ are irrational with probability one for all $k$, and in particular, $x_k^{(1)}$ is nonzero for all $k$ and hence $f$ is differentiable
at all points $x_k$.

Let us summarize the underlying assumptions for all the results in this section.
\begin{assumption}
Let $f$ be defined by \eqref{fdef}, with $a \geq \sqrt{n-1}$,
  and define $x_{k+1}=x_k+t_k d_k$, with $d_k=-\grad f(x_k)$, and with
$t_k$ defined by Algorithm~1, $k=1,2,3,\ldots$, where $x_k=(x_k^{(1)},x^{(2)}_k, \hdots, x^{(n)}_k)$, and $x_0=(x^{(1)}_0,x^{(2)}_0, \hdots, x^{(n)}_0)$
is randomly generated from the normal distribution. All statements in this section are
understood to hold with probability one.
\end{assumption}

\begin{lemma} 
\label{t_a_w}
Suppose $\tau \le 0$ and suppose $|x_k^{(1)}| > a$. Define
\beq
   r_k = \ceil*{ \log_2 {\dfrac{|x_k^{(1)}|}{a}}} \quad \mathrm{so~that} \quad a2^{r_k-1} < |x_k^{(1)}| < a2^{r_k} . \label{rdef}
\eeq
Then, $t_k= 2^{r_k}$. 
\end{lemma} 
\begin{proof} 
Since $|x_k^{(1)}| > a$ any steplength $t\le |x_k^{(1)}|/a$ satisfies  $A(t)$ but fails $W(t)$. 
 Starting with $t =1$, the ``while" loop in Algorithm  \ref{alg1} will carry out $r_k$ doublings of $t$ until $ t > |x_k^{(1)}|/a$, i.e., $W(t)$ holds. 
Hence, in the beginning of stage $r_k +1$, we have  $\alpha=2^{{r_k}-1}$ (a lower bound on $t_k$), $t=2^{r_k}$ and $\beta = +\infty $.
At this point, $t$ satisfies $W(t)$ and since $\tau \le 0$, it also satisfies \eqref{armu}, i.e. $A(t)$. So $t_k = 2^{r_k}$.
\end{proof}

\begin{theorem}
\label{uk_le_a}
Suppose $\tau \le 0$ and $|x^{(1)}_0| > a$. Then after $j\leq r_0$ iterations we have  $|x_j^{(1)}|<a$, where $r_0$ is defined by \eqref{rdef},
and furthermore, for all subsequent iterations, the condition $|x_k^{(1)}|<a$ continues to hold.
\end{theorem}
\begin{proof}
For any $k$ with $|x_k^{(1)}|  > a$ we know from the previous lemma that $t_k = 2^{r_k}$ with $r_k > 0$. From \eqref{wolfe3} and \eqref{rdef} we get
\beq
 	 |x_{k+1}^{(1)}| =at_k - |x_k^{(1)}| < a2^{r_k} - a 2^{r_k-1} = a 2^{r_k-1}.     \label{ukp1} 
\eeq
See Figure~\ref{fig:armk} for an illustration with $n=2$, with $x_k^{(1)}>0$, so \mbox{$-a2^{r_k-1} < x_{k+1}^{(1)} < 0$}. 
Hence, either $|x_{k+1}^{(1)}|  < a$, or $a < |x_{k+1}^{(1)}| < a 2^{r_k-1}$, in which case
from \eqref{rdef} and \eqref{ukp1} we have 
$$          r_{k+1} \leq r_k-1.          $$
So, beginning with $k=0$, $r_k$ is decremented by at least one at every iteration until $|x_k^{(1)}| < a$.
Finally, once $|x_k^{(1)}|<a$ holds, it follows that the initial step $t=1$ satisfies the Wolfe condition $W(t)$, and hence, if $A(t)$ also holds, $t_k$ is set to one,
while if not, the upper bound $\beta$ is set to one so $t_k < 1$. Hence, the next value $x_{k+1}^{(1)}=x_k^{(1)}-\sgn(x_k^{(1)})a t_k$ also satisfies $|x_{k+1}^{(1)}| < a$.
\end{proof}

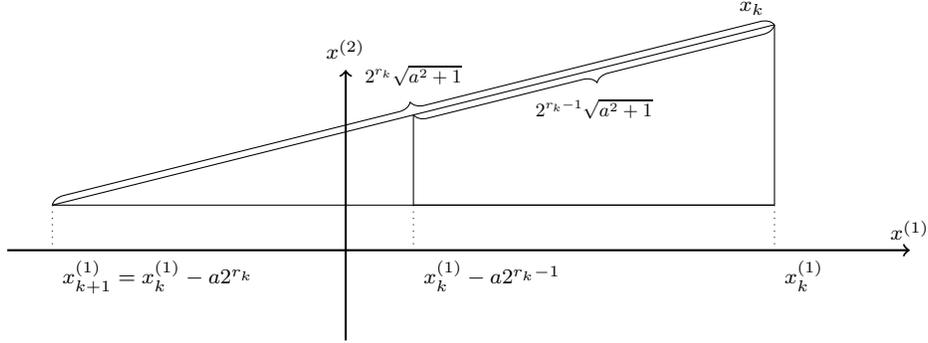
\begin{figure} 
\centering
\begin{turn}{0}
\begin{tikzpicture}[scale=0.3]
\tikzset{anchor=west,font={\fontsize{8pt}{11}\selectfont}}
\draw[thick,<-]  (0,6)node[anchor= south  ]{$x^{(2)}$} -- (0,-6);
\draw[thick,->]  (-15,-2) -- (25,-2)node[anchor= south  ]{$x^{(1)}$};
\draw (19,8) node[anchor= south east]{$x_k$}-- (19,0);
\draw[dotted](19,0)-- (19,-2)node[anchor=north west]{$x_k^{(1)}$};
\draw (3,4)--(3,0);
\draw[dotted](3,0) --(3,-2)node[anchor=north west]{$x_k^{(1)}-a2^{r_k-1}$};
\draw (3,0)-- (19,0)-- (-13,0) ;
\draw[dotted] (-13,0) --(-13,-2) node[anchor=north west]{$ x_{k+1}^{(1)} =x_k^{(1)}-a2^{r_k}$} ;
\draw  (19,8) -- (-13,0) ;
\draw [decorate,decoration={brace,amplitude=5pt}] (-13,0) -- (19,8) node [black,midway,above, scale=0.8,yshift=10pt] {\footnotesize $2^{r_k}\sqrt{a^2+1}$};
\draw [decorate,decoration={brace,amplitude=5pt,mirror}] (3,4) -- (19,8) node [black,midway,below, scale=0.8,yshift=-10pt] {\footnotesize $2^{r_k-1}\sqrt{a^2+1}$};
\end{tikzpicture}
\end{turn}
\caption{Doubling $t$ in order to satisfy $W(t)$.} \label{fig:armk}
\end{figure}

Theorem \ref{uk_le_a} shows that for any $\tau\leq 0$ and sufficiently large $k$ using Algorithm~1 we always have $|x_k^{(1)}| < a$. 
In the reminder of this section we provide further details on the step $t_k$ generated when $|x_k^{(1)}| < a$.
In this case, the initial step $t=1$ satisfies $W(t)$ but not necessarily $A(t)$.  So  Algorithm~1 will repeatedly halve $t$, until it satisfies $A(t)$. 
See Figure \ref{fig:triangle2} for an illustration. 

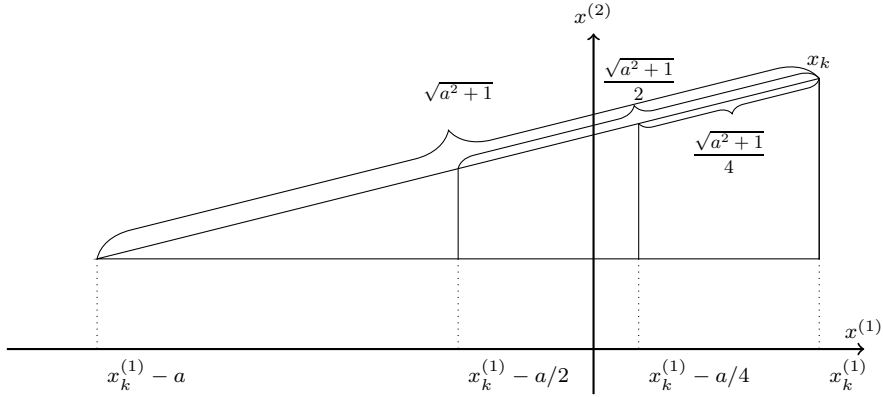
\begin{figure} 
\centering
\begin{turn}{0}
\begin{tikzpicture}[scale=0.6]
\tikzset{anchor=west,font={\fontsize{8pt}{11}\selectfont}}
\draw[thick,<-]  (-2,5)node[anchor= south  ]{$x^{(2)}$} -- (-2,-3);
\draw[thick,->]  (-15,-2)-- (4,-2)node[anchor= south  ]{$x^{(1)}$} ;
\draw (3,4)--(3,0);
 \draw[dotted] (3,0) -- (3,-2) node[anchor=north west]{$x_k^{(1)}$};

\draw (3,0)-- (3,4)node[anchor=south]{$x_{k}$} -- (-13,0)node[anchor=south]{}--(3,0);

\draw[dotted] (-13,0)--(-13,-2) node[anchor=north west]{$x_k^{(1)}-a$} ;
\draw [decorate,decoration={brace,amplitude=15pt,mirror}] (3,4) -- (-13,0) node [black,midway,below, scale=0.8,yshift=45pt] {\footnotesize $\sqrt{a^2+1}$};

\draw  (-5,2) -- (-5,0);
\draw[dotted] (-5,0) --(-5,-2) node[anchor=north west]{$x_k^{(1)}-a/2$};
\draw [decorate,decoration={brace,amplitude=7pt,mirror}] (3,4) -- (-5,2) node [black,midway,below, scale=0.8,yshift=35pt] {\footnotesize $\dfrac{\sqrt{a^2+1}}{2}$};

 \draw  (-1,3) -- (-1,0);
\draw[dotted] (-1,0) --(-1,-2) node[anchor=north west]{$x_k^{(1)}-a/4$};
\draw [decorate,decoration={brace,amplitude=5pt}] (3,4) -- (-1,3) node [black,midway,below, scale=0.8,yshift=-10pt] {\footnotesize $\dfrac{\sqrt{a^2+1}}{4}$};

\end{tikzpicture}
\end{turn}
\caption{Halving $t$ in order to satisfy $A(t)$. }\label{fig:triangle2}
\end{figure}

Suppose for the time being that $\tau=0$ and define $p_k$ by
\beq \label{pk}
     p_k = \ceil*{ \log_2  {\dfrac{a}{|x_k^{(1)}|}}} \quad \mathrm{so~that} \quad \dfrac{a}{2^{p_k}} < |x_k^{(1)}| < \dfrac{a}{2^{p_k-1}}.
\eeq 
For example, in Figure \ref{fig:triangle2}, $p_k = 2$. So, $a/4 < |x_k^{(1)}| < a/2$. Hence $t = 1/2$ satisfies $W(t)$. 
In fact it also satisfies $A(t)$, because for $\tau = 0$, we have 
$$\dfrac{(1+\tau)at}{2} = \dfrac{a}{4} < |x_k^{(1)}|,$$ 
which is exactly the Armijo condition \eqref{armu}. So, Algorithm~1 returns $t_k =1/2$.

On the other hand if we had $ \tau \le -1/2$, $t=1$ would have satisfied the Armijo condition  \eqref{armu} since
$$
\dfrac{ (1+\tau)a}{2} \le \dfrac{a}{4} <  |x_k^{(1)}|. $$
By taking $\tau$ into the formulation we are able to compute the exact value of $t_k$ in the following theorem.

\begin{theorem}\label{min1t}
Suppose $\tau \le 0 $ and $|x_k^{(1)}| < a$.
Then $t_k = \min(1,1/2^{q_k-1})$, where 
\[
               q_k = \ceil*{ \log_2  {\dfrac{(1+\tau)a}{|x_k^{(1)}|}}},  
\]
so 
\beq\label{aulog}
\dfrac{(1+\tau)a}{2^{q_k}} < |x_k^{(1)}| < \dfrac{(1+\tau)a}{2^{q_k-1}}.
\eeq
Note that, unlike $r_k$ and $p_k$, the quantity $q_k$ could be zero or negative.

\end{theorem}

\begin{proof}
If $|x_k^{(1)}| > (1+\tau)a/2$, then $t=1$ satisfies the Armijo condition \eqref{armu} as well as the Wolfe condition, so $t_k$ is set to 1.
Otherwise, $q_k > 1$, so $1/2^{q_k-1} < 1$ and 
Algorithm~1 repeatedly halves $t$ until $A(t)$ holds.
We now show that the first $t$ that satisfies $A(t)$ is such that $|x_k^{(1)}| < at$, i.e., it satisfies $W(t)$ as well.  
Since $\tau \le 0$, the second inequality in \eqref{aulog} proves that steplength $t =1/2^{q_k-1} $ satisfies $W(t)$. Moreover, the first inequality is the Armijo condition  \eqref{armu} with the same steplength.
Furthermore, the second inequality in \eqref{aulog} also shows that $t' = 2t = 1/2^{q_k-2}$ is too large to satisfy the Armijo condition  \eqref{armu}.
Hence $t =1/2^{q_k-1} $ is the first steplength satisfying both $A(t)$ and $W(t)$.
So,  Algorithm 1 returns $t_k= 1/2^{q_k-1}$. 
\end{proof}

Note that if $\tau=0$, $p_k$ and $q_k$ coincide, with $p_k \geq 1$ since $|x_k^{(1)}|<a$, and hence $t_k=1/2^{p_k-1}\leq 1.$
Furthermore, $p_k=1$ and hence $t_k=1$ 
when $a/2 < |x_k^{(1)}| < a$.
\begin{corollary} 
\label{alg1_tau0}
Suppose $\tau=0$. Then $x_k$ converges to a limit $\bar x$ with $\bar x^{(1)}=0$.
\end{corollary}
\begin{proof} 
Assume that $k$ is sufficiently large so that $ | x_k^{(1)}| < a $. 
From \eqref{pk} we have $ a/2^{p_k} < | x_k^{(1)}| $.
Using Theorem~\ref{min1t} we have $t_k = 1/2^{p_k-1}$ and therefore 
\[ 
       | x_{k+1}^{(1)}| = at_k - | x_k^{(1)}| <  \frac{a}{2^{p_k-1}} - \frac{a}{2^{p_k}} = \frac{a}{2^{p_k}}
\]
(see Figure \ref{fig:triangle2} for an illustration). So $p_{k+1} \ge p_k+1$.
 Using Theorem~\ref{min1t} again we conclude  $t_{k+1} \le  1/2^{p_k}$ and so $t_{k+1} \le t_k/2 $. The same argument holds for all subsequent iterates so  $S_N=\sum_{k=0}^{N-1} t_k$ is bounded above as  $N\to\infty$. The result therefore follows from Theorem \ref{key_ineq}.

\end{proof}
\begin{corollary} \label{tau_half}
If $\tau \le -0.5 $ then eventually $t_k = 1$ at every iteration, and $f(x_k)\to -\infty$.
\end{corollary} 
\begin{proof}
As we showed in Theorem \ref{uk_le_a}, for sufficiently large $k$,  $| x_k^{(1)}| < a $ and therefore $t=1$ always satisfies the Wolfe condition, so $t_k \le 1$.
 If  $| x_k^{(1)}|>(1+\tau)a/2$, then $t=1$ also satisfies the Armijo condition \eqref{armu}, so $t_k=1$. 
If $| x_{k+1}^{(1)}|> (1+\tau)a/2$ as well, then $t_{k+1}=1$ and hence $ x_{k+2}^{(1)}= x_k^{(1)}$. It follows that $t_j=1$ for all $j>k+1$. Hence, by Theorem \ref{key_ineq}, $f(x_k)\to -\infty$. Otherwise, suppose $| x_k^{(1)}| < (1+\tau)a/2$ (in case $| x_k^{(1)}| > (1+\tau)a/2$ and $| x_{k+1}^{(1)}| < (1+\tau)a/2$ 
just shift the index by one so that we have $| x_{k-1}^{(1)}| > (1+\tau)a/2$ and $| x_k^{(1)}| < (1+\tau)a/2$).  

 Since $| x_k^{(1)}| < (1+\tau)a/2$, from the definition of $q_k$ in  \eqref{aulog}  we conclude that $2 \le q_k$, i.e. $ 1/2^{q_k-1} \le 1/2 $, so from Theorem~\ref{min1t} we have $t_k=1/2^{q_k-1}\le 1/2$. Since $| x_k^{(1)}| < (1+\tau)a/2^{q_k-1} $ and $1+\tau \le 1/2 $ we have
\beq
      | x_k^{(1)}| < \frac{a}{2^{q_k}}.   \label{ukbd}
\eeq
So by \eqref{wolfe3}
\beq
                   | x_{k+1}^{(1)}|= at_k - | x_k^{(1)}| \ge  \dfrac{a}{2^{q_k-1}} - \dfrac{a}{2^{q_k}} =  \dfrac{a}{2^{q_k}} >       \dfrac{(1+\tau)a}{2^{q_k-1}}  \label{ukp1bd}  
\eeq
and using \eqref{aulog} again we conclude  $q_{k+1} \le q_k -1.$ So, 
\[
        t_{k+1} = \min\left(1, \frac{1}{2^{q_{k+1}-1}}\right) \geq  \min\left(1, \frac{1}{2^{q_k-2}}\right) = \frac{1}{2^{q_k-2}} = 2t_k 
\]
and therefore, applying this repeatedly, after a finite number of iterations, say at iteration ${\bar k}$, we must have  $t_{\bar k} =1$ for the first time.  Furthermore, from \eqref{ukbd} and \eqref{ukp1bd} we have $ | x_k^{(1)}| <| x_{k+1}^{(1)}| $, and applying this repeatedly as well we have $ | x_{\bar k}^{(1)}| < | x_{\bar k+1}^{(1)}| $.  From the Armijo condition \eqref{armu} at iteration ${\bar k }$ we have $ (1+\tau)a/2 \le | x_{\bar k }^{(1)}|$ and therefore
$$     \frac{(1+\tau)a}{2} <  | x_{\bar k+1 }^{(1)}|. $$
Hence, $t=1$ also satisfies the Armijo condition  \eqref{armu} at iteration $\bar k+1$.  With $t_{\bar k} =1$ and $t_{\bar k + 1} = 1$ , we conclude $ x_{\bar k+2}^{(1)}= x_{\bar k}^{(1)}$. It follows that $t_j = 1$ for all $j> \bar k+1$. Hence $f(x_k)\to -\infty$ by Theorem \ref{key_ineq}.
\end{proof}

\color{black}
\section{Experimental Results}\label{sec:expt_results}
In this section we again continue to assume that $f$ and $d_k$ are defined by \eqref{fdef} and \eqref{ddef} respectively. For simplicity we also assume that $n=2$, writing $u=x^{(1)}$ and $v=x^{(2)}$ for convenience.
Our experiments confirm the theoretical results presented in the previous sections and provide some additional insight. 
We know from Theorem~\ref{key_ineq} that when the gradient algorithm fails, i.e, $x_k$ converges to a point $(0,\bar v)$,
the step $t_k$ converges to zero.  However, an implementation of Algorithm~1 in floating point arithmetic must terminate the
``while" loop after it executes a maximum number of times. We used the \matlab\ implementation in {\sc hanso}\footnote{www.cs.nyu.edu/overton/software/hanso},
which limits the number of bisections in the ``while" loop to 30. 

Figure \ref{fig:suc_fai} shows two examples of minimizing $f$ with $a=2$ and $a=5$, with $c_1=0.1$ in both cases, and hence with $\tau<0$ and $\tau > 0$,  respectively.  Starting from the same randomly generated point, we have $f(x_k)\to-\infty$ (success) when $\tau<0$ and $x_k\to$ $(0,\bar v)$ (failure) when $\tau>0$. 
\begin{figure} 
    \centering
    \subfloat[]{\includegraphics[scale=0.5]{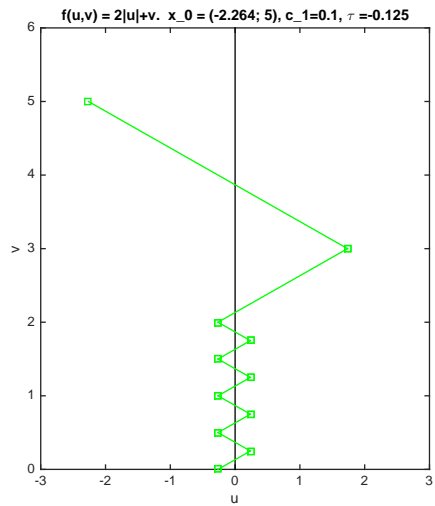}} 
    \subfloat[]{\includegraphics[scale=0.5]{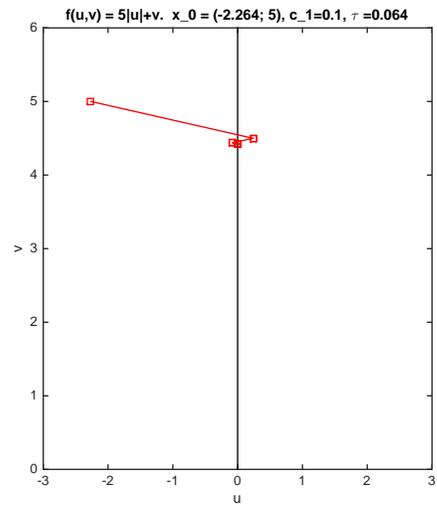} }
    \caption{ Minimizing $f$ with $n=2$, $u=x^{(1)}$, $v=x^{(2)}$ and $c_1=0.1$. \textbf{Left}, with $a = 2$, so $\tau < 0$ and $f(u_k,v_k) \rightarrow -\infty$ (success).
                   \textbf{Right}, with $a = 5$, so $\tau>0$ and $ (u_k,v_k) \rightarrow (0,\bar v)$ (failure).}
    \label{fig:suc_fai}
\end{figure}

For various choices of $a$ and $c_1$ we generated 5000 starting
points $x_0=(u_0,v_0)$, each drawn from the normal distribution with mean 0 and variance 1,
and measured how frequently ``failure" took place, meaning that the line search failed to find an Armijo-Wolfe point within 30 bisections.
If failure did not take place within 50 iterations, i.e., with $k\leq 50$, we terminated the gradient method declaring success.
Figure \ref{fig:afail} shows the failure rates  when (top) $c_1$ is fixed to $0.05$ and $a$ is varied and  
(bottom) when $a = \sqrt 2$ and  $c_1$ is varied. 
Both cases confirm that when $\tau > 0$ the method always fails, as predicted by Corollary \ref{tau_implications},
while when $\tau\le -0.5$, failure does not occur, as shown in Corollary \ref{tau_half}.

\begin{figure} 
    \centering
    \includegraphics[scale=0.63]{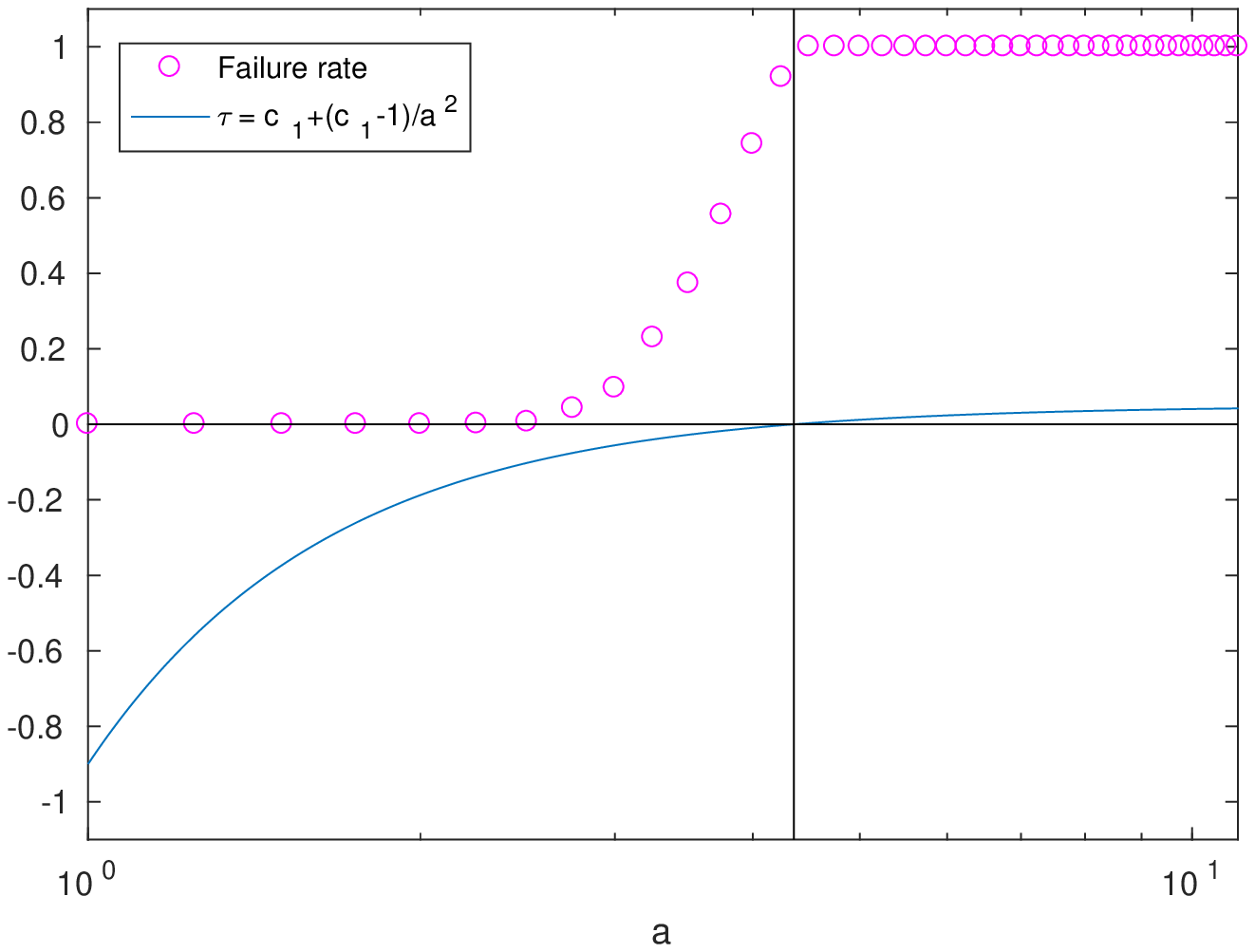} 
    \includegraphics[scale=0.66]{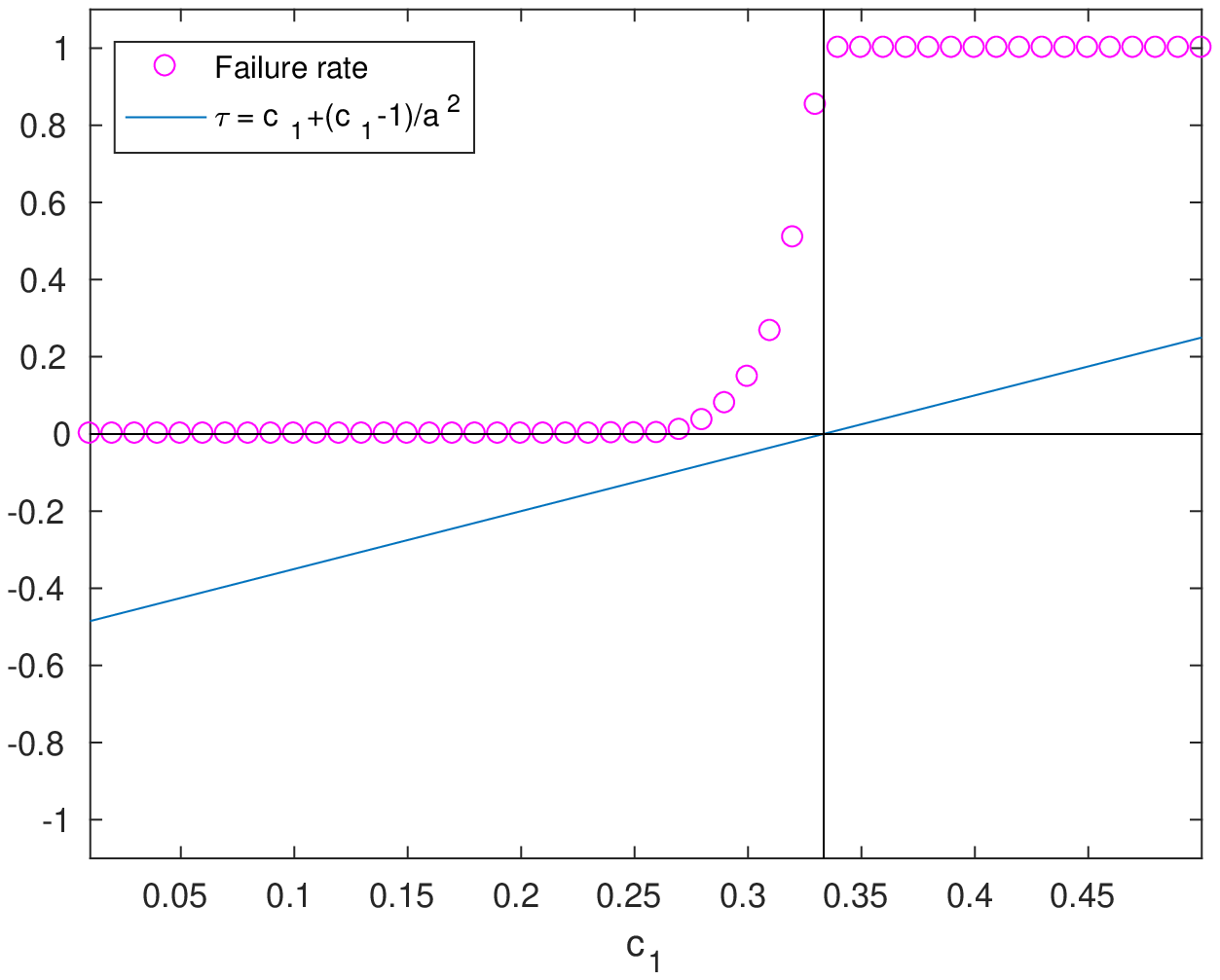} 
    \caption{Failure rates (small circles) for $f$ with $n=2$ when  (top) $c_1$ is fixed to $0.05$ and $a$ is varied and
                   (bottom) $a$ is fixed to $\sqrt 2$ and $c_1$ is varied. 
                   The solid curves show the value of $\tau$. Each experiment was repeated 5000 times.
                   }
    \label{fig:afail}
\end{figure}

\begin{figure}
\centering
\includegraphics[scale=0.4]{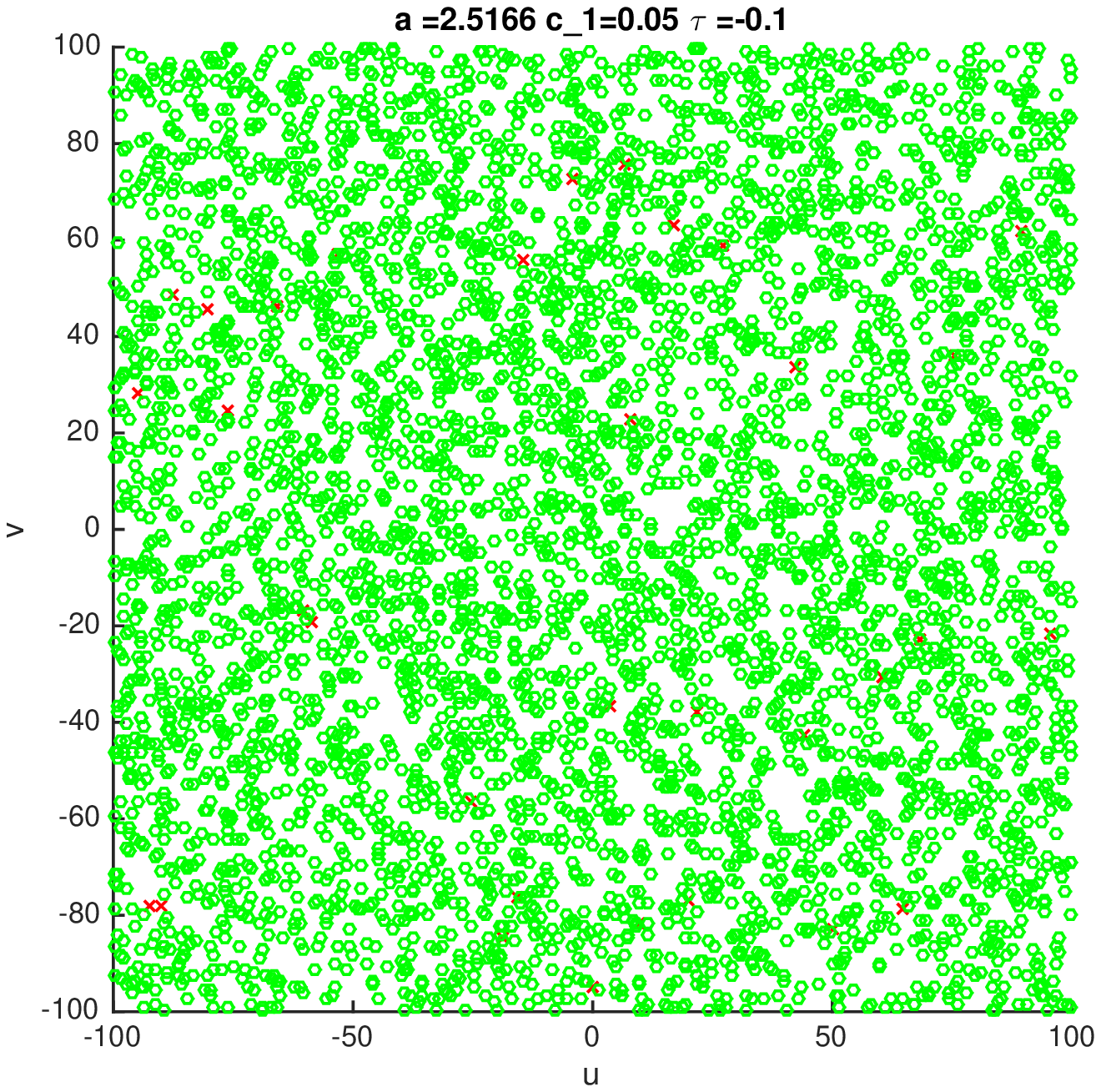}\\
\includegraphics[scale=0.4]{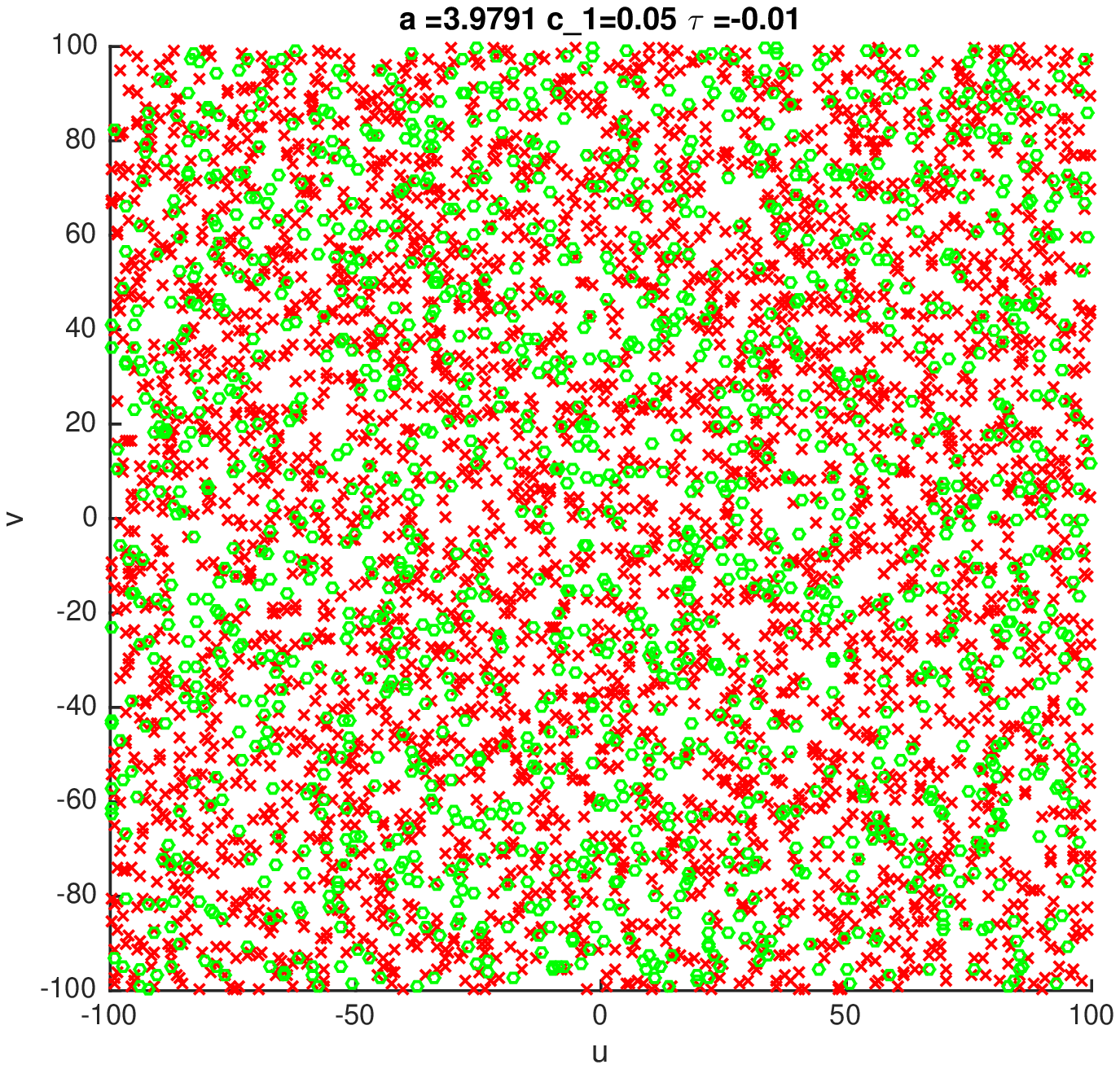}\\
\includegraphics[scale=0.4]{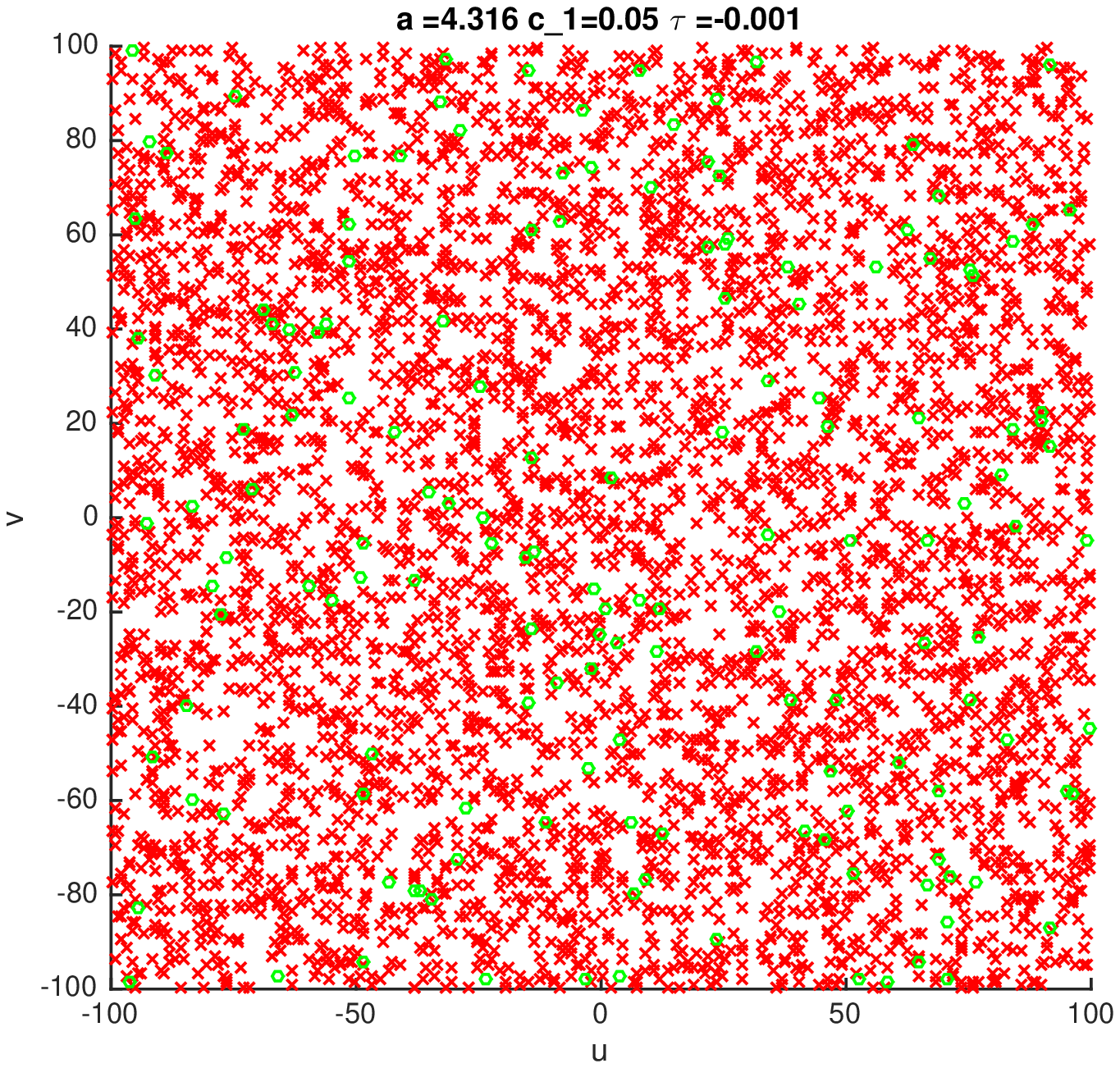}
\caption{Mixed success and failure when $\tau$ is negative but close to zero. Each plot shows 5000 points.
The green circles show starting points for which the method succeeded, generating $x_k =(u_k,v_k)
\in \R^2$ for which $f(x_k)$ is apparently
unbounded below, while the red crosses show starting points for which the method failed, generating $x_k$ converging to
a point on the $v$-axis.}
\label{fig:scat}
\end{figure}

As Figure  \ref{fig:afail}  shows, when $\tau<0$ with $|\tau|$ small, the method may or may not fail, with failure more likely the closer $\tau$ is to zero.
Further experiments for three specific values of $\tau$, namely $-0.1, -0.01$ and $-0.001$, using a fixed value of $c_1=0.05$ and $a$ defined
by $a = \sqrt{(1-c_1)/(c_1 -\tau)}$, confirmed that failure is more likely the closer that $\tau$ gets to zero and also showed that the set of initial points 
from which failure takes place is complex; see Figure~\ref{fig:scat}. The initial points were drawn uniformly from the box $(-100,100)\times(-100,100)$.

We know from Corollary \ref{alg1_tau0}  that, for $\tau=0$, with probability one $t_k\to 0$, so even if high precision
were being used, for sufficiently large $k$ an implementation in floating point must fail.
It may well be the case that failures for $\tau < 0$ occur only 
because of the limited precision being used, and that with sufficiently high precision, these failures would be eliminated.  
This suggestion
 is supported by experiments done reducing the maximum number of bisections to 15, for which the number of failures for $\tau<0$ increased significantly,
and increasing it to 50, for which the number of failures decreased significantly.


\section{Relationship with Convergence Results for Subgradient Methods} \label{sec:subgradient}

Let $h$ be any convex function. The 
subgradient method \cite{SHOR,BERT}
applied to $h$ is a generalization of the gradient method, where $h$ is not assumed to be differentiable at the iterates $\{x_k\}$
and hence, instead of setting $-d_k=\grad h(x_k)$, one defines $-d_k$ to be any element of the
subdifferential set 
\[
       \partial h(x_k) = \big \{g: h(x_k+z) \geq h(x_k) + g^T z ~\forall z \in \R^n \big \}.
 \]
The steplength $t_k$ in the subgradient method is not determined dynamically, as in an Armijo-Wolfe line
search, but according to a predetermined rule. The advantages of the subgradient method with predetermined
steplengths are that it is robust, has low iteration cost, and has a well established convergence theory that does not
require $h$ to be differentiable at the iterates $\{x_k\}$, but the disadvantage is that convergence is usually slow. 
Provided $h$ is differentiable at the iterates, the subgradient method reduces to the gradient method with the
same stepsizes, but it is not necessarily the case that $f$ decreases at each iterate.

We cannot apply the convergence theory of the subgradient method directly to our function $f$ defined in
\eqref{fdef}, because $f$ is not bounded below.
However, we can argue as follows. Suppose that $\tau > 0$, so that we know (by Corollary~\ref{tau_implications})
that for all $x_0$ with $x_0^{(1)}\neq 0$, the iterates $x_k$ generated 
by the gradient method with Armijo-Wolfe steplengths applied to $f$ converge to a point $\bar x$ with $\bar x^{(1)}=0$.
Fix any initial point $x_0$ with $x_0^{(1)}\neq 0$, and let $M = f(\bar x)$, where $\bar x$ is the 
resulting limit point (to make this well
defined, we can assume that the Armijo-Wolfe bracketing line search of Section~\ref{sec:specific_linesch} is in use).
Now define
\[
\tilde f(x) = \max \Big (M-1, a|x^{(1)}| + \sum_{i=2}^{n} x^{(i)} \Big).
\]
Clearly, the iterates generated by the gradient method with Armijo-Wolfe steplengths initiated at $x_0$
are identical for $f$ and $\tilde f$, with $f$ (equivalently, $\tilde f$) differentiable at all iterates $\{x_k\}$,
and with $f(x_k)=\tilde f(x_k) \to M$.
Furthermore, the theory of subgradient methods applies to $\tilde f$. 
One well-known result states that provided the steplengths $\{t_k\}$
are square-summable (that is, $\sum_{k=0}^\infty t_k^2 < \infty$, and hence the steps are ``not too long"),
but not summable (that is, $\sum_{k=0}^\infty t_k = \infty$, and hence the steps are ``not too short"),
then convergence of $\tilde f(x_k)$ to the optimal value $M-1$ must take place \cite{NB01}. Since this does \emph{not}
occur, we conclude that the Armijo-Wolfe steplenths $\{t_k\}$ do \emph{not} satisfy these conditions.
Indeed, the ``not summable" condition is exactly the condition $S_N\to\infty$, where $S_N=\sum_{k=0}^{N-1} t_k$, and
Theorem~\ref{key_ineq} established that the converse, that $S_N$ is bounded above, is equivalent to the function values 
$f(x_k)$ being bounded below.  This, then, is consistent with
the convergence theory for the subgradient method, which says that the steps must not be ``too short"; in
the context of an Armijo-Wolfe line search, when $c_1$ is not sufficiently small, and hence $\tau > 0$,
the Armijo condition is too restrictive: it is causing the $\{t_k\}$ to be ``too short" and hence summable.

Of course, in practice, one usually optimizes functions that are bounded below, but one hopes that a method applied
to a convex function that is not bounded below will not converge, but will generate points $x_k$ with $f(x_k)\to -\infty$.
The main contribution of our paper is to show that, in fact, this does not happen for a simple well known method
on a simple convex nonsmooth function, \emph{regardless of the starting point,}
unless the Armijo parameter is chosen to be sufficiently small --- how small, one does not
know without advance information on the properties of $f$.

\color{black}

\section{Concluding Remarks}\label{sec:conclusion}

Should we conclude from the results of this paper that, if the gradient method with an Armijo-Wolfe line search
is applied to a nonsmooth function, the Armijo parameter $c_1$ should be chosen to be small? 
Results for a very ill-conditioned convex nonsmooth function $\hat f$ devised 
by Nesterov \cite{YN16} suggest that the answer is yes.
The function is defined by
\[
           \hat f(x) = \max\{|x_1|, |x_i-2x_{i-1}|, i=2,...,n\}.
 \]
 Let $\hat x_1=1, \hat x_i = 2\hat x_{i-1} + 1, i=2,...,n$.
 Then $\hat f(\hat x) = 1 = \hat f(\mathbb{1})$ although $\|\hat x\|_\infty \approx 2^n$ and $\|\mathbb{1}\|_\infty=1$,
so the level sets of $\hat f$ are very ill conditioned.
 The minimizer is $x=0$ with $\hat f(x)=0$. Figure~\ref{fig:yuri_les_houches} shows 
 function values computed by applying five different methods to minimize $\hat f$ with $n=100$.
 The five methods are: the subgradient method with $t_k=1/k$, a square-summable but not summable sequence that
 guarantees convergence;
 the gradient method using the Armijo-Wolfe bracketing line search of Section~\ref{sec:specific_linesch};
the limited memory BFGS method \cite{NW06} with 5 and 10 updates respectively (using ``scaling"); and the full BFGS method \cite{NW06, LO13}; the BFGS variants also use the  
 same Armijo-Wolfe line search.\footnote{In our implementation, we made no
 attempt to determine whether $\hat f$ is differentiable at a given point or not. This is essentially
 impossible in floating point arithmetic, but
 as noted earlier, the gradient is defined at randomly generated
 points with probability one; there is no reason to suppose that any of the methods tested will
 generate points where $\hat f$ is not differentiable, except in the limit, and hence the ``subgradient"
 method actually reduces to the gradient method with $t_k=1/k$.
 See \cite{LO13} for further discussion.}
The top and bottom plots in Figure~\ref{fig:yuri_les_houches} show
 the results when the Armijo parameter $c_1$ is set to 0.1 and to $10^{-6}$ respectively.
 The Wolfe parameter was set to $0.5$ in both cases.
 These values were chosen to satisfy the usual
 requirement that $0<c_1<c_2<1$, while
 still ensuring that $c_1$ is not so tiny that it is effectively zero in floating point arithmetic.
All function values generated by the methods are shown, including those evaluated in the line search.
 The same initial point, generated randomly, was used for all methods; the results
using other initial points were similar. 

\begin{figure} 
    \centering
    \includegraphics[scale=0.3]{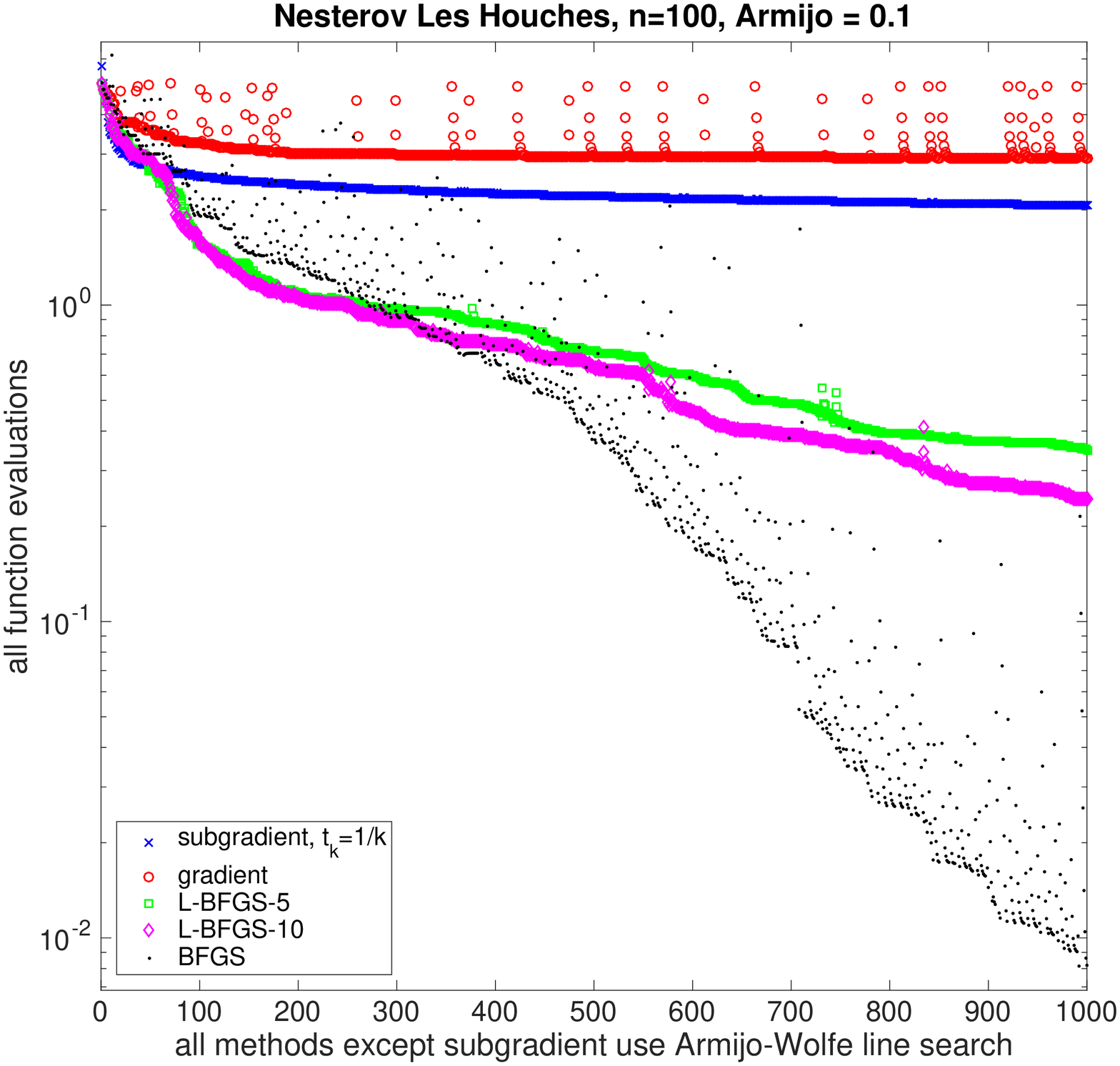}
    \includegraphics[scale=0.3]{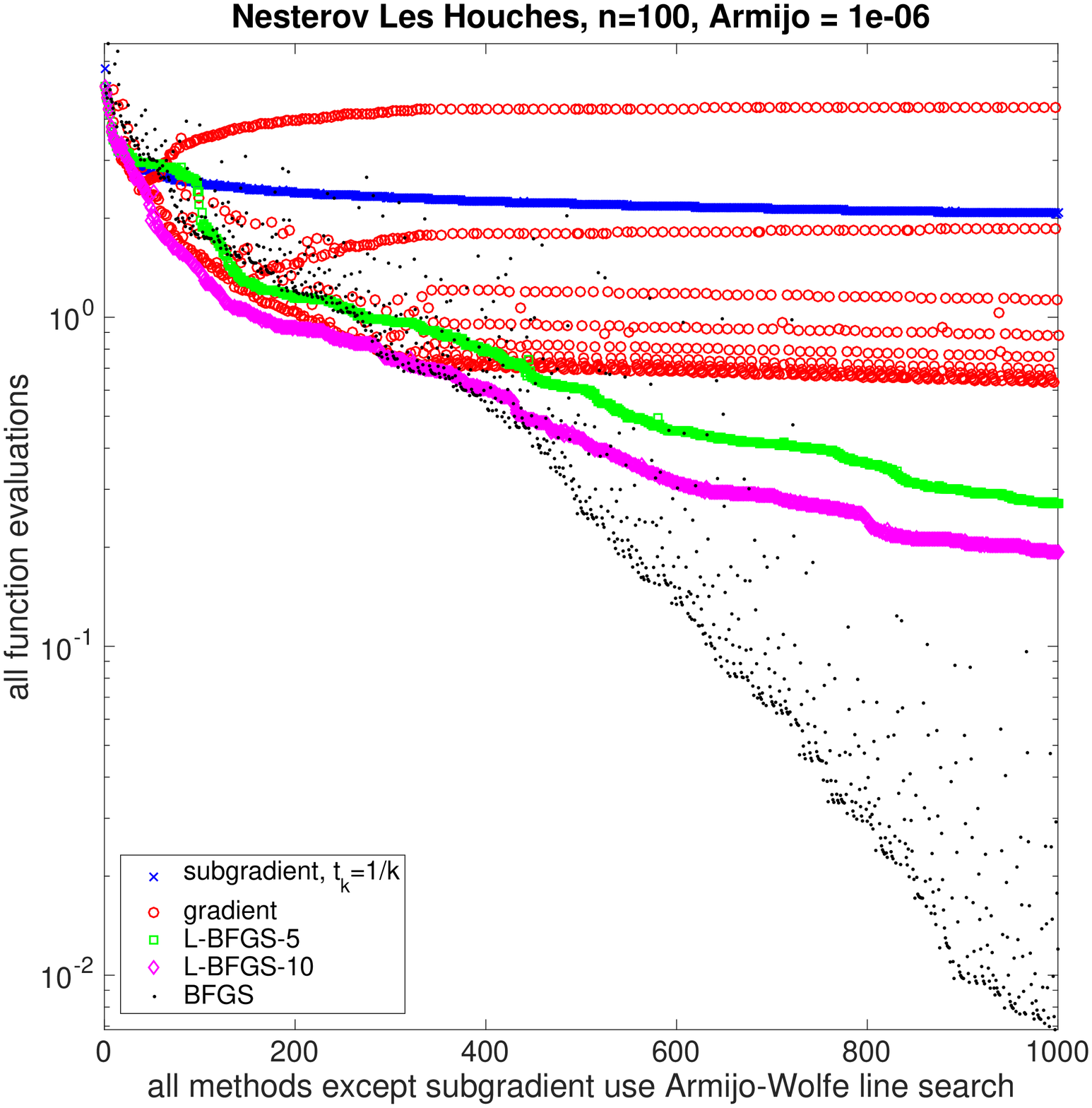}
    \caption{Comparison of five methods for minimizing Nesterov's ill-conditioned convex nonsmooth function 
    $\hat f$. The subgradient method (blue crosses) uses $t_k=1/k$. The gradient, limited-memory BFGS (with 5 and 10 updates
    respectively) and full BFGS methods (red circles, green squares, magenta diamonds and black dots)
    all use the Armijo-Wolfe bracketing line search. All function evaluations are shown.
    Top: Armijo parameter $c_1=0.1.$ Bottom:   Armijo parameter $c_1=10^{-6}$.            }
    \label{fig:yuri_les_houches}
\end{figure}

\bigskip
\bigskip
For this particular example, we see that, in terms of reduction of the function value within a given number of evaluations,
 the gradient method with the Armijo-Wolfe line search when the
Armijo parameter is set to $10^{-6}$ performs better than using the subgradient method's predetermined sequence $t_k=1/k$,
but that this is not the case when the Armijo parameter is set to $0.1$. The smaller value allows the gradient method to take
steps with $t_k=1$ early in the iteration, leading to rapid progress, while the larger value forces shorter steps, quickly leading
to stagnation. Eventually, even the small Armijo parameter requires many steps in the line search --- one can see that on the right
side of the lower figure, at least 8 function values per iteration are required.  
One should not read too much into the results for one example,
but the most obvious observation from Figure \ref{fig:yuri_les_houches} is that the full BFGS and
limited memory BFGS methods are much more effective than the gradient or subgradient methods. 
This distinction becomes far more dramatic if we run the methods
for more iterations: BFGS is typically able to reduce $\hat f$ to about $10^{-12}$ in about 5000 function evaluations, 
while the gradient and subgradient methods fail to reduce $\hat f$ below $10^{-1}$ in the same number of function evaluations.
The limited memory BFGS methods consistently perform better than the gradient/subgradient methods but worse than full BFGS.
The value of the Armijo parameter $c_1$ has little effect on the BFGS variants. 

These results are consistent with substantial prior experience with applying the full
BFGS method to nonsmooth problems, both convex and nonconvex
\cite{LO13,CMO17,GLO17,GL18}. 
However, although the BFGS method requires far fewer operations per iteration than bundle methods or
gradient sampling, it is still not practical when $n$
is large. Hence, the attraction of limited-memory BFGS which, like the gradient and subgradient methods,
requires only $O(n)$ operations per iteration. In a subsequent
paper, we will investigate under what conditions the limited-memory BFGS method applied to
the function $f$ studied in this paper might generate iterates that converge to a non-optimal point,
and, more generally, how reliable a choice it is for nonsmooth optimization.

\appendix
\bibliographystyle{alpha}
\bibliography{refs}

\end{document}